\documentclass[12pt]{amsart}
\usepackage{amsmath,amsfonts,amssymb,amsthm}

\newcommand{\mgn}{\overline{\mathcal{M}}_{g,n}}
\newcommand{\mgno}{\mathcal{M}_{g,n}}
\newcommand{\mg}{\overline{\mathcal{M}}_g}
\newcommand{\mgo}{\mathcal{M}_g}

\newcommand{\supp}{\mathrm{supp}}

\newcommand{\Z}{\mathbb{Z}}
\newcommand{\C}{\mathbb{C}}
\newcommand{\M}{\mathcal{M}}
\newcommand{\sieg}{\mathbb{H}}

\newcommand{\A}{\mathcal{A}}
\newcommand{\even}{\mathcal{E}}

\newcommand{\Sp}[1]{\Simp(g,\Z)}

\newcommand{\F}{\mathbb{F}_2}
\newcommand{\e}{\varepsilon}

\DeclareMathOperator{\diag}{diag}
\DeclareMathOperator{\Simp}{Sp}
\DeclareMathOperator{\Id}{Id}
\DeclareMathOperator{\proj}{Proj}

\newtheorem{Theorem}{Theorem}
\newtheorem{Proposition}{Proposition}
\newtheorem{Lemma}{Lemma}
\newtheorem{Corollary}{Corollary}

\newtheorem{Remark}{Remark}

\begin{document}
\title{An affine open covering of $\mathcal{M}_g$ for $g \le 5$}
\author{Claudio Fontanari and Stefano Pascolutti}
\date{}

\begin{abstract}
We prove that the moduli space $\mathcal{M}_g$ of smooth curves of genus $g$ is the union of $g-1$ affine open subsets for every $g$ with $2 \le g \le 5$, as predicted by an 
intriguing conjecture of Eduard Looijenga.
\end{abstract}

\maketitle

\section{Introduction}

A purely algebro-geometric approach to the cohomology of the coarse moduli spaces $\mgno$ and $\mgn$ of smooth (respectively, stable) curves of genus $g$ with $n$ marked points has been recently developed by Enrico Arbarello and Maurizio Cornalba in the two papers \cite{ArbCor:98}, \cite{ArbCor:08}, where the only essential result borrowed from geometric topology is a vanishing theorem due to John Harer. Namely, the fact that $H_k(\mgno)$ vanishes for $k > 4g-4+n$ if $n>0$ and for $k > 4g-5$ if $n=0$ was deduced in \cite{Harer:86} from the 
construction of a $(4g-4+n)$-dimensional spine for $\mgno$ by means of Strebel differentials. On the other hand, it is conceivable that Harer's vanishing is only the tip of an iceberg of deeper geometrical properties (see \cite{HaiLoo:98}, Problem (6.5)).  For instance, a conjecture of Eduard Looijenga says that $\mgo$ is a union of $g-1$ affine open subsets (see \cite{FabLoo:99}, Conjecture~11.3), but (as far as we know) until now there have been no advances in this direction. Notice that Looijenga's conjecture trivially holds for $g=2,3$: indeed, 
it is well-known that $\mathcal{M}_2$ is affine and that non-hyperelliptic curves of genus $3$ can be canonically embedded as quartic plane curves, hence $\mathcal{M}_3 = \mathcal{M}_3 \setminus \{\text{the hyperelliptic locus}\} \cup \mathcal{M}_3 \setminus \{\text{the locus of plane quartics with at least one hyperflex}\}$ is the union of two affine open subsets. 

Here we are going to prove the following 

\begin{Theorem}\label{main} 
For every $g$ with $2 \le g \le 5$ the moduli space $\mgo$ is the union of $g-1$ affine open subsets. 
\end{Theorem} 

We point out that, under the same assumptions on $g$, from the properties of the linear system of quadrics passing 
through the canonical image of a smooth projective genus $g$ curve it follows that $\mgo$ admits an affine stratification 
of depth $g - 2$ (see \cite{FL}). Our approach in the present note relies instead on the theory of modular forms.

\section{Notation and preliminaries}

We work over the complex field $\mathbb{C}$ and we denote by $\sieg_g$ the Siegel upper half-space of symmetric complex matrices 
with positive-definite imaginary part, the so-called period matrices. The action of the symplectic group $\Sp{g}$ on $\sieg_g$ 
is given by
\begin{equation}\label{ABCD}
\begin{pmatrix} 
A&B\\ C&D\end{pmatrix}\circ\tau:= (A\tau+B)(C\tau+D)^{-1},
\end{equation}
where the elements of $\Sp{g}$ are thought as four $g\times g$ blocks and they preserve the symplectic form given in the block form 
as $\left(\begin{smallmatrix} 0& 1\\ -1& 0 \end{smallmatrix}\right)$. We recall that the quotient $\A_g = \sieg_g / \Sp{g}$ has the structure of a quasi-projective variety and it can be viewed as the moduli space of principally polarized abelian varieties. 
Let $\mgo$ and $\mathcal{H}_g$ be the moduli spaces of smooth curves and hyperelliptic curves of genus $g$, respectively. 
It is a very well known fact that
\begin{equation}\label{torelli}
\mathcal{H}_g \hookrightarrow \mgo \hookrightarrow \A_g.
\end{equation}
Moreover, we denote by $\Gamma_g:=\Sp{g}$ the integral symplectic group and
we define the principal congruence subgroup $\Gamma_g[2] \subseteq \Gamma_g$:
\[
	\Gamma_g[2] = \{M \in \Gamma_g \mid M \equiv \Id_{2g} \mod 2\},
\]
which acts on $\sieg_g$ in the same way as $\Gamma_g$ does. The action of $\Gamma_g[2]$ on $\A_g$ induces a level $2$ structure: 
namely, we denote by $\A_g[2] = \sieg_g / \Gamma_g[2]$ the moduli space of principally polarized abelian varieties with a level $2$ structure. Since we have a map $\A_g[2] \to \A_g$, we can define $\mgo[2]$ as the preimage of $\mgo$ and $\mathcal{H}_g[2]$ as the preimage of $\mathcal{H}_g$ in $\A_g[2]$. The analogue of \eqref{torelli} holds for a level $2$ structure:
\[
	\mathcal{H}_g[2] \hookrightarrow \mgo[2] \hookrightarrow \A_g[2].
\]
For a period matrix $\tau\in\sieg_g$, $z\in \C^g$ and $\e,\delta\in \F^g$ (where $\F$ denotes the abelian group $\Z/2 \Z = \{0,1\}$, 
for which we use the additive notation) the associated theta function with characteristic $m=[\e, \delta]$ is
\[
 \theta_m(\tau, z) =\sum_{n\in\Z^g} \exp \left(\pi i ((n+\e/2)'\tau (n+\e/2)+ 2(n+\e/2)'( z+\delta/2)
 \right),
\]
where we denote by $X'$ the transpose of $X$. As a function of $z$, $\theta_m(\tau, z)$ is odd or even depending on whether the scalar product $\e\cdot\delta\in\F$ is equal to 1 or 0, respectively. Theta constants are restrictions of theta functions to $z=0$. We shall write $\theta_m$ for theta constants. It is easy to check that odd theta constants are identically 0, since they are the valuation in 0 of odd functions.

\begin{Remark}\label{split}
Let $\tau \in \sieg_g$ be a period matrix of the form $\tau = \left(\begin{smallmatrix}\tau_1 & 0\\ 0 & \tau_2 \end{smallmatrix}\right)$, with $\tau_1 \in \sieg_{g_1}$ and $\tau_2 \in \sieg_{g_2}$, $g_1 + g_2 = g$. We split $m = [\e, \delta] \in \F^{2g}$ as $m_1 \oplus m_2$, where $m_1 = [\e_1, \delta_1] \in \F^{2g_1}$, $m_2 = [\e_2, \delta_2] \in \F^{2g_2}$ and $m = [\e_1 \e_2, \delta_1 \delta_2]$; then we have
	\[
		\theta_m(\tau) = \theta_{m_1}(\tau_1)\cdot \theta_{m_2}(\tau_2).
	\]
\end{Remark}

For a set of characteristics $M=(m_1, m_2,\dots, m_k)$ we let
\[
	P(M):=\prod_{i=1}^k\theta_{m_i}.
\]
A holomorphic function $f\colon \sieg_g\to\C$ is a modular form of weight $k/2$ with respect to a subgroup $\Gamma\subset\Gamma_g$ of finite index if
\[
	f(\gamma\circ\tau)=\det(C\tau+D)^{k/2}f(\tau) \qquad \forall\gamma\in\Gamma,\forall\tau\in\sieg_g,
\]
where $C$ and $D$ are as in (\ref{ABCD}), and if additionally $f$ is holomorphic at all cusps for $g=1$. 
We denote by $[\Gamma, k/2]$ the space of such functions, which turns out to be a finite dimensional vector space. 
Moreover,
\[
		 A(\Gamma):=\bigoplus_{k=0} ^{\infty}[\Gamma, k/2]
\]
is a finitely generated ring. The associated projective variety $\A_g^* = \proj(A(\Gamma))$ is the so called \emph{Satake compactification} of $\sieg_g / \Gamma$.

Theta constants are modular forms of weight $k/2$ with respect to the subgroup $\Gamma(4,8) \subseteq \Gamma_g[2]$ of matrices $M = \left(\begin{smallmatrix}A&B\\C&D\end{smallmatrix}\right) \equiv 1 \mod 4$ and $\diag(AB')\equiv \diag(CD')\equiv 0 \mod 8$ (see \cite{Igu:64}). 
For further use, we denote by $\Gamma_g (1,2)$ the subgroup of $\Gamma_g$ defined by 
$\diag(AB')\equiv \diag(CD')\equiv 0 \mod 2$; note that theta constants will have an automorphy factor with respect to $\Gamma_g(1,2)$, while their eighth power is modular form.

We need the following rough formulation of a classical result by Mumford about the hyperelliptic locus (for a more precise statement we refer to \cite{Mum:84}).
\begin{Theorem}\label{mumford}
	Let $\tau \in \sieg_g$. Then $\tau$ is the period matrix of a smooth hyperelliptic curve if and only if exactly
	\[
		v(g) = 2^{g-1}(2^g+1) - \frac{1}{2}\binom{2g+2}{g+1}
	\]
suitable even theta constants vanish at the point $\tau$. Each suitable sequence of theta constants defines an irreducible component of $\mathcal{H}_g[2]$. 
The full modular group acts transitively on the components.
\end{Theorem}
Moreover, we recall the structure of $\mathcal{H}_g[2]$ (see \cite{Tsu:90}).
\begin{Remark}\label{Tsu}
For $g \leq 2$, $\mathcal{H}_g[2]$ is an irreducible variety. For $g \geq 3$, $\mathcal{H}_g[2]$ breaks into disjoint irreducible components isomorphic to each other. The number of components of the hyperelliptic locus is
	\[
		2^{g(2g+1)}\prod_{k=1}^g \frac{1-2^{-2k}}{(2g+2)!}.
	\]
Hence, we have 36 components for $g = 3$, each defined by the vanishing of a single even theta constant, 13056 components for $g=4$, each defined by the vanishing of 10 suitable even theta constants, 51806208 components for $g=5$, each defined by the vanishing of 66 suitable even theta constants.
\end{Remark}

We introduce four special modular forms which will give us the affine open covering we are looking for. 
We denote by $\even$ the subset of $\F^{2g}$ of even characteristics; it is easy to show that $|\even| = 2^{g-1}(2^g + 1)$. Let
\begin{align}
	F_\text{null} &= P(\even) = \prod_{m \in \even} \theta_m;\label{thetanull}\\
	F_1 &= \sum_{m \in \even} \left(P(\even)/\theta_m\right)^8 ;\label{una}\\
	F_H &= \sum_{A \subseteq \F^{2g}} (P(\even\setminus A))^8;\label{hyperelliptic}\\
	F_T &= 2^g \sum_{m \in \even} \theta_m^{16} - \left(\sum_{m \in \even} \theta_m^8 \right)^2\label{trigonal}.
\end{align}
In \eqref{hyperelliptic}, $A$ varies among all suitable sets of $v(g)$ characteristics 
corresponding to the irreducible components of $\mathcal{H}_g[2]$ as in Theorem \ref{mumford}. 
We need to take the eighth power of the theta constants in order to obtain the modularity 
of the above forms with respect to the modular group.

\begin{Remark}
The modular form $F_\text{null}$ has weight $2^{g-2}(2^g+1)$. Its vanishing locus is a divisor $\Theta_\text{null}$ on $\A_g$.
\end{Remark}

\begin{Remark}
The modular form $F_1$ has weight $2^{g+1}(2^g+1)-4$. It defines a divisor $D_1$ on $\A_g$.
\end{Remark}

\begin{Remark}
The modular form $F_H$ has weight $2\binom{2g+2}{g+1}$ and it coincides with $F_\text{null}^8$ for $g = 2$, 
since no theta constant vanishes on the hyperelliptic locus. In the case $g = 3$, the modular 
form $F_H$ coincides with $F_1$ since every component of the hyperelliptic locus is characterized by the vanishing 
of a single theta constant. Let $D_H$ be the divisor defined by $\{F_H = 0\}$ on $\A_g$.
\end{Remark}

\begin{Remark}\label{trigonal_locus}
The modular form $F_T$ has weight $8$ and it is not identically 0 only for $g \geq 4$. Indeed, for $g = 2,3$ it vanishes 
identically on $\sieg_g$, while for $g = 4$ it vanishes on the preimage of $\M_4$ in $\sieg_4$. When $g \geq 4$, it defines a divisor $D_T$ on $\A_g$, which coincides with the closure of $\M_4$ when $g = 4$ and with the closure of the trigonal locus when $g = 5$ (see \cite{GruSal:09}).
\end{Remark}

In order to handle the divisors defined by modular forms on $\A_g^*$, we will make use of the following fact.

\begin{Lemma}\label{ampleness}
For $g \geq 3$ all modular forms define ample divisors on $\A_g^*$.
\end{Lemma}
\begin{proof}
For $g \geq 3 $ the group of the Weil divisors of $\A_g^{*}$ modulo principal divisors is isomorphic to $\Z$, hence 
a suitable multiple of any effective divisor $D$ in $\A_g^*$ is very ample. By \cite{Db}, Lemma 2.1, $D$ is ample. 
Since every divisor defined by a modular form is effective, our claim follows.
\end{proof}

Let $D$ be a divisor defined by a modular form. Since $\M_g$ contains complete curves when $g \geq 3$ (indeed, the Satake compactification is projective and has boundary of codimension $2$ for $g \ge 3$), we have that $D \cap \M_g \neq \emptyset$. Hence each of the previously described divisors either contains $\M_g$ or defines a divisor in $\M_g$. In the latter case, we use the same notation for 
the induced divisor inside $\M_g$.

Next, we prove a fundamental result about $F_H$ (for further details, see also 
\cite{Salvati}, Theorem 2).

\begin{Lemma}\label{hyperavoid}
The modular form $F_H$ never vanishes on $\mathcal{H}_g$.
\end{Lemma}

\begin{proof}
Let $\tau$ be the period matrix of a hyperelliptic curve. By Theorem~\ref{mumford}, there is a suitable set $A$ of $v(g)$ theta constants which vanish at $\tau$. 
Hence all terms but $P(\even \setminus A)^8$ contain at least one of the vanishing theta constants and 
$F_H(\tau) = P(\even \setminus A)^8(\tau) \neq 0$. It follows that $F_H$ never vanishes on $\mathcal{H}_g$.
\end{proof}

\section{The main result}

Now we can exhibit an explicit open covering of $\mgo$ for every $2 \leq g \leq 5$. 
We have already recalled the description for $g = 2,3$ in the Introduction, hence here we focus on the cases 
$g = 4$ and $g = 5$. Namely, we are going to prove that
\begin{align}
	\M_4 &= \M_4 \setminus \Theta_\text{null} \cup \M_4 \setminus D_1 \cup \M_4 \setminus D_H, \label{M4},\\
	\M_5 &= \M_5 \setminus \Theta_\text{null} \cup \M_5 \setminus D_1 \cup \M_5 \setminus D_H \cup \M_5 \setminus D_T. \label{M5}
\end{align}

Our proof relies on two key ideas. The first one is a straightforward application of the Cornalba-Harris ampleness criterion. 
(see \cite{CorHar:88}).

\begin{Proposition} \label{affine}
Let $D$ be an effective divisor on $\mathcal{M}_g$ and let $\overline{D}$ 
be its closure in the Deligne-Mumford compactification $\overline{\mathcal{M}}_g$. 
If $[\overline{D}]= a \lambda - \sum b_i \delta_i$ with $a, b_i > 0$, then 
$\mgo \setminus \supp(D)$ is an affine open subset. 
\end{Proposition}

\begin{proof}
Just notice that $E = a \lambda - \sum b_i \delta_i + \sum (b_i - \varepsilon) \delta_i = a \lambda - \varepsilon \delta$ with $\varepsilon > 0$ small enough is an effective divisor on $\overline{\mathcal{M}}_g$ such that $\overline{\mathcal{M}}_g \setminus 
\supp(E) = \mathcal{M}_g \setminus \supp(D)$ and $E$ is ample by Cornalba-Harris ampleness criterion \cite{CorHar:88}.
\end{proof}

Proposition~\ref{affine} yields the following useful result.

\begin{Corollary}\label{affcor}
Let $D$ be an effective divisor on $\mgo$, let $\tilde{D}$ be its closure in the Satake compactification $\A_g^{*}$ and $\overline{D}$ be its closure in the Deligne-Mumford compactification $\mg$.
If $\tilde{D}$ contains the product of periods of smooth curves and periods of nodal curves, then $\overline{D} = a \lambda - \sum b_i \delta_i$ with $a, b_i > 0$. In particular, $\mgo \setminus \supp(D)$ is affine.
\end{Corollary}

\begin{proof}
There is a standard map from the Deligne-Mumford compactification $\mg$ to the Satake compactification $\mgo^*$ (i.e. the closure of $\M_g$ in the Satake compactification) which takes boundary divisors of $\mg$ to products of periods of smooth curves and periods of nodal curves. Thus in the notation of Proposition \ref{affine} we have $b_i > 0$, since $\tilde{D}$ contains the image of the whole boundary of $\mg$, therefore any function which vanishes on $\overline{D}$ vanishes on every $\delta_i$ with positive multiplicity, and now our claim follows from Proposition~\ref{affine}.
\end{proof}
\begin{Remark}
We could avoid the Cornalba-Harris ampleness criterion by observing that a modular form always induces an ample divisor $\tilde{D}$ on $\M_g^*$ and it defines a divisor $D$ on $\M_g$. Obviously $\M_g^*\setminus \tilde{D}$ is affine. Now, whenever $\tilde{D} = D \cup (\M_g^*\setminus \M_g)$ we have that $\M_g \setminus D = \M_g^*\setminus \tilde{D}$ is affine.
\end{Remark}

Next we present a nice criterion to check whether $F_H$ vanishes on $\tau$.
\begin{Lemma}\label{crit}
Let $\tau \in \sieg_g$. If there exist more than $v(g)$ even theta constants vanishing on $\tau$, then $F_H(\tau) = 0$.
\end{Lemma}

\begin{proof}
This is just a trivial application of the pigeon hole principle. Indeed, each summand of $F_H$ is the product of $\frac{1}{2}\binom{2g+2}{g+1}$ even theta constants out of $2^{g-1}(2^g+1)$ total even theta constants, since 
there are exactly $v(g)$ theta constants left out of the product. If there are more than $v(g)$ theta constants 
vanishing at $\tau$, then every summand contains at least one of them, hence it vanishes at $\tau$.
\end{proof}

We shall also apply the following auxiliary result, which is essentially Lemma 3 in \cite{Accola}. A self-contained proof is reproduced here for readers' convenience.

\begin{Lemma} \label{accola}
If a curve $C$ of genus $5$ with a base point free $g^1_3$ 
carries two half-canonical $g^1_4$, then $C$ is hyperelliptic.
\end{Lemma}

\begin{proof}
We claim (see \cite{Accola}, Lemma 2) that if $C$ is a curve of genus $5$ with a base point free
$g^1_3$, then every half-canonical $g^1_4$ has a fixed point. Indeed, let $x+y+z$ be a divisor 
in the $g^1_3$ with three distint points and notice that $h^0(C, K_C-x-y-z)= h^0(C,K_C-x-y) = 
h^0(C,K_C-x-z)$. If $D_x$ and $D_y$ are two divisors in the half-canonical $g^1_4$ containing $x$ 
and $y$, respectively, it follows that $z$ is contained in the canonical divisor $D_x+D_y$, say in 
$D_x$. Now, if $y$ is not a base point of the $g^1_4$, then there is a divisor $D$ in the $g^1_4$ 
which does not contain $y$. On the other hand, $y$ has to be contained in the canonical divisor 
$D_x+D$ containing $x$ and $z$, hence $D_x =x+y+z+w$ and $g^1_4 = g^1_3 + w$, as claimed.  
By the claim, both half-canonical $g^1_4$ have a fixed point, namely, the first one is $g^1_3 + x$ 
and the second one is $g^1_3 + y$ with $x \ne y$. We have $2g^1_3 + 2x = \vert K_C \vert = 
2 g^1_3 + 2y$, hence $2x \sim 2y$ with $x \ne y$ and $C$ turns out to be hyperelliptic. 
\end{proof}

\begin{proof}[Proof of Theorem \ref{main}]
For $g = 2,3$ the statement is trivial, as recalled in the Introduction, hence we need to check it only 
for $g = 4,5$. Let first $g = 4$. According to \eqref{M4}, our three open sets are the following:
	\[
		\M_4 = (\M_4\setminus \Theta_\text{null}) \cup (\M_4 \setminus D_1) \cup (\M_4\setminus D_H).
	\]
We show that the above divisors satisfy the hypotheses of Corollary \ref{affcor}. For $D_1$ and $\Theta_\text{null}$ this is obvious. Hence we just need to check that the closure $\tilde{D}_H$ of $D_H$ in $\A_4^*$ contains the product of periods of smooth curves and periods of nodal curves, and then apply Corollary \ref{affcor}. If $\tau \in \tilde{D}_H$ is a product of periods, i.e. $\tau = \left(\begin{smallmatrix}\tau_1 & 0 \\ 0 & \tau_2\end{smallmatrix}\right)$, with $\tau_1 \in \M_{g_1}$ and $\tau_2 \in \M_{g_2}$, 
$g_1 + g_2 = 4$, then by Remark~\ref{split} we have $\theta_m(\tau) = 0$ if $m = m_1\oplus m_2$ with $m_1 \in \F^{2g_1}$ and 
$m_2 \in \F^{2g_2}$ odd characteristics. There are two possible cases. If $g_1 = g_2 = 2$, then we have $6 \cdot 6 = 36 > 10 = v(4)$ even characteristics which split as odd$\oplus$odd in the notation of Remark \ref{split}; by Lemma \ref{crit}, $F_H$ vanishes on $\tau$. If $g_1 = 1$ and $g_2 = 3$, then we have $1\cdot 28 = 28 > 10 = v(4)$ even characteristics and again $F_H$ vanishes on $\tau$. The analogue result holds for nodal curves. Hence $\mathcal{M}_4 \setminus D_H$ is an affine open set. 
 Moreover, $\Theta_\text{null} \cap D_1$ is the hyperelliptic locus in $\mathcal{M}_4$. Indeed, set-theoretically
 \begin{equation}
 	\label{twotheta}
	 \Theta_\text{null} \cap D_1 = \bigcup_{m_1 \neq m_2} \{\theta_{m_1}=\theta_{m_2} = 0\}
 \end{equation}
and the intersection of this locus with $\M_4$ gives the hyperelliptic locus (see \cite{Igu:81}). By Lemma \ref{hyperavoid} we have $\mathcal{M}_4 \supset \Theta_\text{null} \cap D_1 \cap D_H = \emptyset$, hence \eqref{M4} holds.

Let now $g = 5$. According to \eqref{M5}, our four open sets are the following:
	\[
		\M_5 = ( \M_5 \setminus \Theta_\text{null}) \cup (\M_5 \setminus D_1) \cup (\M_5 \setminus D_H) \cup (\M_5 \setminus D_T).
	\]
Again, we check that all involved divisors satisfy the hypotheses of Corollary \ref{affcor}. For $D_1$ and $\Theta_\text{null}$ this is obvious. Next, we claim that the closure $\tilde{D}_H$ of $D_H$ in $\A_5^*$ contains the product of periods of smooth curves and periods of nodal curves. Indeed, it is sufficient to prove that if $\tau = \left(\begin{smallmatrix}\tau_1 & 0 \\ 0 & \tau_2\end{smallmatrix}\right) \in \M_{g_1}\times \M_{g_2}$, with $g_1 + g_2 = 5$, 
then more than $v(5) = 66$ theta constants vanish on $\tau$. If $g_1 = 1$ and $g_2 = 4$, then we have $1 \cdot 120 = 120 > 66$ 
even theta constants vanishing on $\tau$. If $g_1 = 2$ and $g_2 = 3$, then we have $6\cdot 28 = 168 > 66$ even theta constants 
vanishing on $\tau$. By Lemma \ref{crit}, $F_H$ vanish on every product of smooth curves. The analogue result holds for nodal curves. Hence $\mathcal{M}_4\setminus D_H$ is an affine open set. 
Finally, we obtain \eqref{twotheta} as before and together with Lemma \ref{accola} we conclude that $\Theta_\text{null} \cap D_1 \cap D_T$ is exactly the hyperelliptic locus. By Lemma \ref{hyperavoid} we have $\mathcal{M}_5 \supset \Theta_\text{null} \cap D_1 \cap D_T \cap D_H = \emptyset$, hence \eqref{M5} holds.
\end{proof}

\section{Aknowledgements}

We are grateful to Riccardo Salvati Manni for strongly stimulating our joint project, 
as well as to Enrico Arbarello and Gabriele Mondello for enlightening conversations on 
this research topic.

Our work has been partially supported by GNSAGA of INdAM and MIUR Cofin 2008 - 
Geometria delle variet\`{a} algebriche e dei loro spazi di moduli (Italy).

\hspace{0.5cm}

\noindent
Claudio Fontanari \newline
Dipartimento di Matematica \newline 
Universit\`a degli Studi di Trento \newline 
Via Sommarive 14 \newline 
38123 Trento, Italy. \newline
E-mail address: fontanar@science.unitn.it

\hspace{0.5cm}

\noindent
Stefano Pascolutti \newline
Dipartimento di Matematica ``Guido Castelnuovo'' \newline
Universit\`a di Roma ``La Sapienza'' \newline
P.le Aldo Moro 5 \newline
00185 Roma, Italy. \newline
E-mail address: pascoluttistefano@gmail.com 

\end{document}